\documentclass[11pt,reqno]{amsart}
   \usepackage{amsmath}
 \usepackage{amssymb}
  \usepackage{amsthm}
    \usepackage[utf8]{inputenc}
      \usepackage[english]{babel}
           \usepackage{graphicx}
             \usepackage{geometry}
\usepackage{indentfirst}

\usepackage{srcltx} 

\usepackage{appendix}
\usepackage{xcolor}

\newtheorem{theorem}{Theorem}[section]
\newtheorem{lemma}{Lemma}[section]

\newtheorem{definition}{Definition}[section]
\newtheorem{remark}{Remark}[section]
\newtheorem{corollary}[lemma]{Corollary}
  \newcommand{\MB}{\mathcal{B}}
\newcommand{\tsigma}{\widetilde{\sigma}}
\newcommand{\MC}{\mathcal{C}}
\newcommand{\vg}{{\boldsymbol{\gamma}}}
\newcommand{\nnu}{{\boldsymbol{\nu}}}
\newcommand{\bbeta}{{\boldsymbol{\beta}}}
\newcommand{\T}{{\boldsymbol{\tau}}}

\newcommand{\Gammawedge}{$\Gamma_{\text{\rm wedge}}$\ }

\newcommand{\gammawedge}{\Gamma_{\text{\rm wedge}}}

\newcommand{\Gammashock}{$\Gamma_{\text{\rm shock}}$\ }
\newcommand{\gammashock}{\Gamma_{\text{\rm shock}}}

\newcommand{\gammasonic}{\Gamma_{\text{\rm sonic}}}

\def\be{\begin{equation}}
\def\ee{\end{equation}}
\def\ba{\begin{array}{lll}}
\def\ea{\end{array}}
\def\bc{\begin{cases}}
\def\ec{\end{cases}}
\def\bbm{\begin{bmatrix}}
\def\ebm{\end{bmatrix}}
\def\bpm{\begin{pmatrix}}
\def\epm{\end{pmatrix}}
\def\bvm{\begin{vmatrix}}
\def\evm{\end{vmatrix}}
\def\bin{\begin{enumerate}}
\def\ein{\end{enumerate}}
\def\bit{\begin{itemize}}
\def\eit{\end{itemize}}
\def\bid{\begin{description}}
\def\eid{\end{description}}

\renewcommand{\div}{\mathrm{div}}

\newcommand{\mr}{\mathbb{R}}

\newcommand{\deq}{:=}

\newcommand{\p}{\partial}

\newcommand{\pSi}{\varphi}

\renewcommand{\r}{\rho}
\newcommand{\g}{\gamma}

\renewcommand{\O}{\Omega}

\renewcommand{\L}{\Lambda}

\newcommand{\Gw}{\Gamma_{\text{\rm wedge}}}

\newcommand{\Gsh}{\Gamma_{\text{\rm shock}}}

\newcommand{\ra}{\rightarrow}

\newcommand{\ttau}{{\boldsymbol{\tau}}}

\numberwithin{equation}{section}
\numberwithin{figure}{section}

\begin{document}
	\title[Loss of regularity of solutions of the Lighthill Problem]{Loss of Regularity of Solutions of the Lighthill
		Problem for Shock Diffraction for Potential Flow}
	\author{Gui-Qiang Chen}
	\address{Gui-Qiang G. Chen, Mathematical Institute, University of Oxford,
		Oxford, OX2 6GG, UK;\,
		AMSS and UCAS, Chinese Academy of Sciences, Beijing 100190, China}
	\email{chengq@maths.ox.ac.uk}
	
	\author{Mikhail Feldman}
	\address{Mikhail Feldman, Department of Mathematics, University of Wisconsin, Madison, WI 53706-1388, USA}
	\email{feldman@math.wisc.edu}
	
	\author{Jingchen Hu}
	\address{Jingchen Hu, School of Mathematical Sciences, University of Science and Technology of China, Hefei, China}
	\email{jchu@mail.ustc.edu.cn}
	
	\author{Wei Xiang}
	\address{Wei Xiang, Department of Mathematics, City University of Hong Kong, Kowloon, Hong Kong, P. R. China}
	\email{weixiang@cityu.edu.hk}
	
	\date{\today}

\begin{abstract}
	We are concerned with the suitability of the main models of compressible fluid dynamics for the Lighthill problem for shock diffraction
	by a convex corned wedge, by studying the regularity of solutions of the problem, which can be formulated as a free boundary problem.
	In this paper, we prove that there is no regular solution that is subsonic up to the wedge corner
	for potential flow.
	This indicates that, if the solution is subsonic
	at the wedge corner, at least a characteristic discontinuity (vortex sheet or entropy wave)
	is expected to be generated,
	which is consistent with the experimental and computational results.
	Therefore, the potential flow equation is not suitable for the Lighthill problem so that
	the compressible Euler system must be considered.
	In order to achieve the non-existence result, a weak maximum principle for the solution is established,
	and several other mathematical techniques are developed.
	The methods and techniques developed here are also useful
	to the other problems with similar difficulties.
\end{abstract}

\keywords{Loss of regularity, free boundary problems, Lighthill problem,
nonlinear equations of second order,
mixed elliptic-hyperbolic type, degenerate elliptic equations,
conservation laws, potential flow equation,
shock diffraction, compressible flow}

\subjclass[2010]{Primary: 35M10, 35M12, 35B65, 35L65, 35L70, 35J70, 76H05, 35L67, 35R35;
Secondary: 35L15, 35L20, 35J67, 76N10, 76L05}
\maketitle

\section{Introduction}
We are concerned with the regularity of solutions of the Lighthill problem
for shock diffraction by a two-dimensional
convex cornered wedge, which is not only a longstanding open problem in fluid mechanics
but also fundamental in the mathematical theory of multidimensional conservation laws,
since the shock diffraction configurations are fundamental for the local structure
of general entropy solutions. The Lighthill problem can be formulated as a free boundary problem.

The main objective of this paper is to address the fundamental issue: which of the main models of compressible
fluid dynamics is suitable for the shock diffraction problem.
There are two main models: the compressible Euler system and the potential flow equation.
We prove that there is no regular solution of the shock diffraction problem
in the framework of potential flow.
In fact, our results show the limitation of the potential flow equation,
since solutions of this equation have a high regularity,
so that it is too rigid for the Lighthill problem.
This shows that the Euler system may be more stable
as it allows nonregular solutions.

The Euler equations for polytropic potential flow consist of
\begin{equation}\label{equ:1}
\partial_t\rho+\mbox{div}_{\mathbf{x}}(\rho\nabla_{\mathbf{x}}\Phi)=0
\end{equation}
with Bernoulli's law:
\begin{equation}\label{equ:2}
\p_t\Phi+\frac{1}{2}|\nabla_{\mathbf{x}}\Phi|^2+i(\r)=B_0,
\end{equation}
where $\mathbf{x}=(x_1,x_2)\in\mr^2$,
$\rho$ is the density, $\Phi$ is the velocity potential such that
$\nabla_{\mathbf{x}}\Phi=:(u,v)=\mathbf{v}$ is the flow velocity,
and $B_0$ is the Bernoulli constant determined by the incoming flow and/or boundary
conditions.
For a polytropic gas, by scaling,
$$
c^2(\r)=\r^{\g-1},\quad
i(\r)=\frac{\r^{\g-1}-1}{\g-1},\qquad \g>1,
$$
where $c(\r)$ is the sound speed.
For the isothermal case, $\g=1$,
$$
p(\rho)=\rho, \qquad  c^2(\rho)=1, \qquad  i(\rho)= ln \rho
$$
as the limiting case when $\g\to 1$.
In this paper, we focus mainly on the case $\gamma>1$, since a similar
argument works when $\gamma=1$.

\begin{figure}
	\centering
	\includegraphics[height=4.5cm]{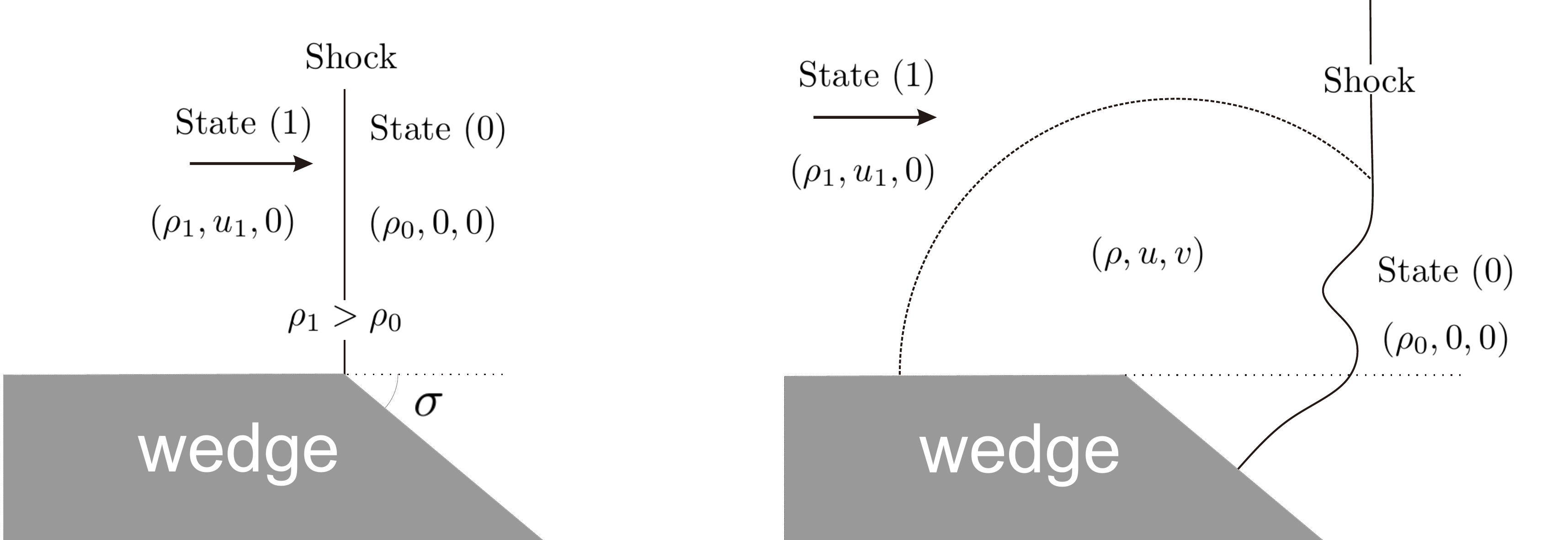}
	\caption{The Lighthill problem for shock diffraction}\label{Fig1}
\end{figure}

As shown in Fig. \ref{Fig1}, we consider two piecewise constant Riemann data
with the left state, State $(1)$:
$(\r_1,u_1,0), u_1>0$, and the right state, State $(0)$: $(\r_0,0,0)$, separated by
a vertical shock $S_0$ and above the wedge with the corner angle $\theta_w=\pi-\sigma$,
where $\sigma\in(0,\pi)$.
As this incident shock passes through the
wedge, the incident shock interacts with the sonic circle and
becomes a transonic shock.
On the other hand, physical observation and numerical analysis {\cite{Skews,ST}} indicate
that, when a shock is diffracted by a convex cornered wedge,
at least a characteristic discontinuity (vortex sheet or entropy wave) should be generated.
Mathematically, to confirm this phenomenon, one needs to show the nonexistence of regular solutions of the Lighthill problem.
Therefore,
we are interested in the question of whether any solution of the Lighthill problem,
governed by the potential flow equation,
is irregular at the wedge corner; that is,
a regular solution, which is Lipschitz at the corner,
does not exist, where the exact notion of regular solutions
will be given in \S \ref{regularsolution}.

In this paper, we prove the nonexistence of regular solutions to
this Riemann data
when the initial left state
$(\rho_1, u_1,0)$ is subsonic, {\it i.e.},
$u_1<c_1$. Then the pseudo-velocity $(U,V)=(u_1,0)$ at the origin is
pseudosubsonic,
and the degenerate boundary is the sonic circle
centered at $(u_1,0)$ with radius $c_1$.
One of the main ingredients in our analysis is to develop a weak maximum principle at the wedge corner
for the velocity in a special direction by the integration method.
It is well known that, if a regular solution is assumed to have a $C^1$-velocity at the wedge corner,
then we can directly apply Hopf's maximum principle to the directional derivative
of the potential function to reach a contradiction.
However, it does not apply directly to the solution that is not $C^2$-continuous.
Therefore, it is the key point in \S \ref{conclusion} to develop the weak maximum
principle without the $C^2$-continuity of the solution.

The mathematical study of the shock diffraction problem dates back to the 1950s by the work
of
Lighthill \cite{Lighthill,Lighthill2} via asymptotic analysis, which is now called the Lighthill problem;
also see Bargman \cite{Bargman},
Fletcher, Weimer, and Bleakney \cite{FWB}, and Fletcher, Taub, and Bleakney \cite{FTB}
via experimental analysis,
as well as
Courant, and Friedrichs \cite{CourantFriedrichs-book1999SupersonicFlowandShockWaves} and Whitham \cite{Whitham}.
To date, all efforts for rigorous mathematical analysis of the Lighthill problem
have focused
on some simplified models.
For one of these models, the nonlinear wave system,
Kim \cite{k-jde20102482906nonlinearwavesystemshockdiffraction} first studied
this problem for the right wedge angle with an additional physical assumption
that the transonic shock does not collide with the sonic circle of the right state.
More recently, in Chen, Deng, and Xiang \cite{cdx-1}, this assumption was removed,
and the existence and optimal regularity of shock diffraction
configurations were established for all angles of the convex wedge
via a different approach.

A closely related problem, shock reflection-diffraction by concave cornered wedges,
has been systematically analyzed
in Chen, and Feldman \cite{CF, CFbook}, Bae, Chen, and Feldman \cite{BCF} and Chen, Feldman, and Xiang \cite{CFXConvexity,CFXUniqueness},
where the existence and uniqueness of regular
shock reflection-diffraction configurations has been established up to
the detachment wedge angle.
For the nonsymmetric case, the non-existence of regular solutions has been
shown in Feldman, and Hu \cite{FH} when the wedge angle is sufficiently close to $\pi$.
The Prandtl-Meyer reflection for supersonic potential flow impinging onto
a solid wedge has also been analyzed first in Elling, and Liu \cite{EllingLiu2}
and, most recently, in Bae, Chen, and Feldman \cite{BCF2,BCF3} for the general case.
For other related references, we refer the reader to Serre \cite{Serre} for Chaplygin gas,
Canic, Keyfitz, and Kim \cite{ckk-SIAMjma20061947} for
the nonlinear wave system, and Zheng \cite{z-actamathsinica200622177pressuregradientsystem}
for the pressure-gradient system. See also \cite{FX,QX} for steady flow passing through a wedge.

The organization of this paper is as follows:
In \S \ref{sec:problem}, we formulate the Lighthill problem as a free boundary problem
for a nonlinear equation of mixed elliptic-hyperbolic type,
prove the important fact that the speed of incident shock is subsonic and
the incident shock hits the sonic circle in the first quadrant,
and then introduce the notion of regular solutions of the free boundary problem
and state the main theorem.
In \S \ref{conclusion},
we prove the monotonicity property of regular solutions
with respect to a special direction.
This means that, if the regular solution exists,
the speed along some direction cannot achieve the negative minimum anywhere in $\overline\Omega$.
Then we reach a contradiction in \S \ref{conclusion} from the monotonicity property,
actually via a weak version of Hopf's maximum principle.
In Appendix A, we prove that a solution of some type of linear elliptic equations,
which is only assumed to be $L^{\infty}$ at the wedge corner, is actually continuous.

\section{Mathematical formulation of the Lighthill problem}\label{sec:problem}
In this section, we first formulate the Lighthill problem as an initial-boundary
value problem ({Problem 2.1}) for the potential flow equation,
then reduce it to a boundary value
problem ({Problem 2.2}) and further to a free boundary problem ({Problem 2.3})
for a nonlinear equation of mixed elliptic-hyperbolic type,
prove that the speed of incident shock is subsonic and
the incident shock hits the sonic circle in the first quadrant,
and finally introduce the notion of regular solutions of the free boundary problem
and state the main theorem.

\subsection{The Lighthill problem}

When a plane shock in the $(t,\mathbf{x})$-coordinates,
$\mathbf{x}=(x_1,x_2)\in\mathbb{R}^2$, with left state $(\rho, u,v)=(\rho_1, u_1,0)$
and right state $(\rho_0,0,0)$ satisfying $u_1>0$ and $\rho_0<\rho_1$, passes a wedge
$$
W:=\{(x_1,x_2)\, :\, x_2<0,\ x_1<x_2\cot\theta_{\rm w}\}
$$
stepping down, the shock diffraction phenomenon occurs.
Mathematically, this problem can be formulated as the
following initial-boundary value problem.

\vspace{-10pt}
\begin{figure}[!h]
  \centering
  \qquad \includegraphics[width=0.30\textwidth]{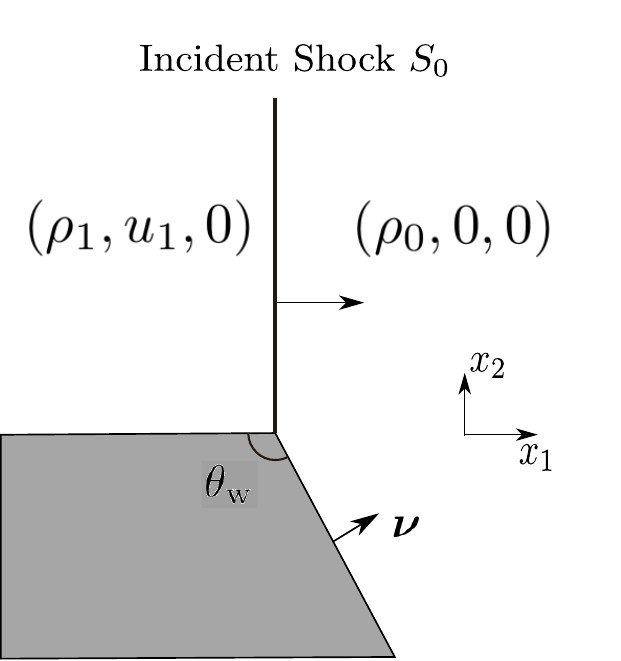}\qquad\qquad\qquad
  \includegraphics[width=0.30\textwidth]{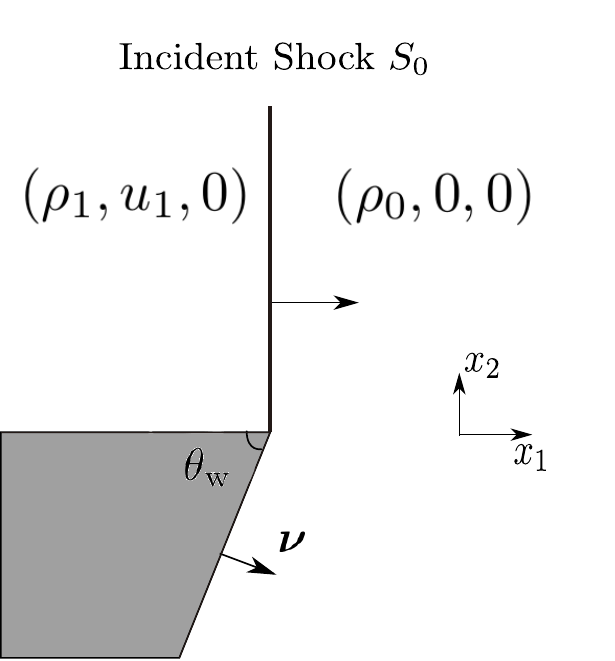}
  \caption{Initial-boundary value problem}
\label{fig:initial}
\end{figure}

{\bf Problem 2.1 }(initial-boundary value problem) (see Fig. \ref{fig:initial}).
Seek a solution of system \eqref{equ:1}--\eqref{equ:2}
with the initial condition at $t=0$,
\begin{equation}\label{con:1.1}
 (\rho,\Phi)|_{t=0}=
 \begin{cases}
 (\rho_1,u_1x_1)\qquad\,\, &\mbox{in}\,\, \{x_1<0, x_2>0\},\\[1mm]
(\rho_0,0) \qquad\,\, &  \mbox{in}\,\,\mathbb{R}^2 \setminus  (\overline{W} \cup \{x_1 \le 0, x_2\ge 0\} ),
\end{cases}
 \end{equation}
and the slip boundary condition along the wedge boundary $\partial W$:
 \begin{equation}\label{con:1.2}
 \nabla\Phi\cdot\nnu=0,
 \end{equation}
where $\nnu$ is the exterior unit normal to $\partial W$.

\medskip
The initial-boundary value problem
\eqref{equ:1}--\eqref{equ:2} and \eqref{con:1.1}--\eqref{con:1.2} is
invariant under the self-similar scaling
$$
(t, \mathbf{x})\rightarrow(\alpha t, \alpha \mathbf{x}),\quad
(\rho,\Phi)\rightarrow(\rho, \frac{\Phi}{\alpha})\qquad\,\,\,
\mbox{for}\,\, \alpha\neq 0.
$$
Thus, we seek self-similar solutions with the form
\begin{equation}\label{1.2a}
\rho(t,\mathbf{x})=\rho(\xi,\eta),\quad
\Phi(t,\mathbf{x})=t\phi(\xi,\eta)
\qquad\,\,\,\mbox{for}\,\,\, (\xi,\eta)=\frac{\mathbf{x}}{t},
\end{equation}
where
$\pSi:=\phi-\frac{\xi^2+\eta^2}{2}$ is called a pseudovelocity potential
with $(\pSi_{\xi},\pSi_{\eta})=(u-\xi,v-\eta)=(U,V)$ as the corresponding
pseudovelocity.
In the self-similar plane, the domain outside the wedge is
\begin{equation}\label{2.8a}
\Lambda:=\mathbb{R}^2 \setminus  \overline{W}.
\end{equation}

Then the pseudovelocity potential function
$\pSi$ satisfies the following equation for self-similar
solutions:
\begin{equation}
\mbox{div}(\rho D\pSi)+2\rho=0
\label{1.1}
\end{equation}
with
\begin{equation}
\frac{1}{2}|D\pSi|^2+\pSi+\frac{c^2}{\gamma-1}=\frac{c^2_0}{\gamma-1}, \label{1.2}
\end{equation}
where the divergence ${\rm div}$ and gradient $D$ are with respect to the
self-similar variables $(\xi,\eta)$ and $c_0=c(\rho_0)$.
Therefore, the potential function
$\pSi$ is governed by the following second-order potential flow equation:
\begin{equation}\label{equ:potential flow equation}
\div(\r(D\pSi,\pSi)D\pSi)+2\r(D\pSi,\pSi)=0
\end{equation}
with
\begin{equation}\label{equ:bernoulli law for density}
\r(D\pSi,\pSi)=\left(c_0^2-(\g-1)\left(\pSi+\frac{1}{2}|D\pSi|^2\right)\right)^{\frac{1}{\gamma-1}}.
\end{equation}
Equation \eqref{equ:potential flow equation} is of
mixed hyperbolic-elliptic type: It is elliptic if
and only if
$
|D\pSi|<c(D\pSi,\pSi):=\rho(D\pSi, \pSi)^{\frac{\gamma-1}{2}},
$
which is equivalent to
\begin{equation}\label{criterion:ellipticity}
|D\pSi|<c_{\star}(\pSi,\g)\deq\sqrt{\frac{2}{\gamma+1}\big(c_0^2-(\gamma-1)\pSi\big)}.
\end{equation}

Since the problem involves shock waves,
the solutions of
\eqref{equ:potential flow equation}--\eqref{equ:bernoulli law for density}
have to be understood as weak solutions
in distributional sense. Based on the argument in \S 2.1 in \cite{CF},
for a piecewise $C^1$ solution $\pSi$ separated by a shock $S$, if $\varphi\in W_{\rm loc}^{1,\infty}(\Lambda)$, it is easy to
verify that $\pSi$ satisfies \eqref{equ:potential flow equation}--\eqref{equ:bernoulli law for density} in the distributional sense
if and only if it is a classic solution of
\eqref{equ:potential flow equation}--\eqref{equ:bernoulli law for density}
in each smooth subregion and satisfies the following
Rankine--Hugoniot conditions across $S$:
\begin{eqnarray}
&&[\r(D\pSi,\pSi)D\pSi\cdot\nnu]_{S}=0,\label{RH conditon involve normal direction}\\
&&[\pSi]_{S}=0,\label{RH conditon continous}
\end{eqnarray}
where $[w]_S$ denote the difference between the right and left traces of quantity $w$ along $S$.

In fact, a discontinuity of $D\pSi$ satisfying the Rankine--Hugoniot
conditions
\eqref{RH conditon involve normal direction}--\eqref{RH conditon continous} is called a shock if it satisfies the following
physical entropy condition: {\it The density function $\r$ increases across a shock	in the pseudoflow direction}.

\subsection{Location of the incident shock and the boundary value problem in the self-similar coordinates}
\label{subsec:incident shock}\label{subsec:sonic angle}
Consider the left state (1):
$(\r,u,v)=(\r_1,u_1,0)$ with $\r_1>0$ and $u_1>0$, and
the right state (0): $(\r,u,v)=(\r_0,0,0)$ with $\r_1>\r_0>0$
such that the entropy condition and the
Rankine--Hugoniot conditions are satisfied on the incident shock separating these two states.

We note that, from \eqref{1.2}, the potentials of states (1) and (0) are
\begin{equation}\label{2.11a}
\pSi_1=-\frac{\xi^2+\eta^2}{2}+u_1\xi +\frac{c_0^2-c_1^2}{\gamma-1}-\frac 12 u_1^2,\qquad
\pSi_0=-\frac{\xi^2+\eta^2}{2}.
\end{equation}

From \eqref{RH conditon continous},
$[D\pSi]_{S}$ is the normal direction to the shock.
Then we have the following two equations as the Rankine-Hugoniot
conditions along the flat shock $L$ separating state (0) from a constant state $(u, v, \rho)$ with potential $\varphi$:
\begin{equation}
\label{RH condition for mass}
u\big(\r(u-\xi)+\r_0\xi\big)+v\big(\r(v-\eta)+\r_0\eta\big)=0
\end{equation}
and
\begin{equation}
\label{RH condition for position}
\pSi=\pSi_0.
\end{equation}

Thus, on the incident shock separating states (1) and (2),
equations (\ref{RH condition for mass})--(\ref{RH condition for position}) hold with $(\r,u,v)=(\r_1,u_1,0)$ and $\varphi=\varphi_1$.
From (\ref{RH condition for position}), we find that the incident shock
is line $\{\xi=\xi_1\}$  for $\xi_1=\frac 12 u_1-\frac{c_0^2-c_1^2}{u_1(\gamma-1)}$.
From \eqref{RH condition for mass} with
$(\r,u,v,\xi)=(\r_1,u_1,0,\xi_1)$, we have
\begin{equation}\label{formula: incident shock speed}
0<u_1=\frac{\r_1-\r_0}{\r_1}\xi_1<\xi_1.
\end{equation}
Then \eqref{RH condition for mass}--\eqref{RH condition for position} imply
\begin{equation}\label{formula: position of incident shock}
\xi_1=\sqrt{\frac{2\r^2_1(c_1^2-c_0^2)}{(\g-1)(\r_1^2-\r_0^2)}}>0.
\end{equation}

We now show that the incident shock interacts with the
sonic circle of the left state $(1)$ through the following
relation:
\begin{equation}\label{relation: shock hit sonic circle}
0<\xi_1-u_1<c_1.
\end{equation}
In fact, the positivity follows directly
from \eqref{formula: incident shock speed}.
From \eqref{formula: incident shock speed}--\eqref{formula: position of incident shock},
we have
$$
c^2_1-(\xi_1-u_1)^2=c_1^2-\frac{2\r_0^2(c_1^2-c_0^2)}{(\g-1)(\r_1^2-\r_0^2)}
=\frac{(\g-1)c_1^2\r_1^2-(\g+1)\r_0^2c_1^2+2\r_0^2c_0^2}{(\g-1)(\r_1^2-\r_0^2)}.
$$
Since $c^2=\r^{\g-1}$, it suffices for the second inequality in
\eqref{relation: shock hit sonic circle}
to prove that
$$
f(s):=2s^{\g+1}-(\g+1)s^2+(\g-1)>0
\qquad\,\,\, \mbox{for $s=\frac{\rho_0}{\rho_1}\in (0,1)$}.
$$
This can be seen by the fact that
$f'(s)<0$ and $f(1)=0$.
Then we have the following lemma.

\begin{lemma}\label{prop2.1 incident shock}
Consider the left state $(\rho_1,u_1,0)$ and right state $(\rho_0,0,0)$
with $\rho_1>\rho_0>0$ and $u_1>0$.
Then the location of the incident shock $\xi_1$ satisfies
\eqref{relation: shock hit sonic circle},
and angle $\theta_1$ determined by $\tan \theta_1=\frac{\eta_1}{\xi_1-u_1}$
for $P_1=(\xi_1, \eta_1)$ must
be in the interval $(0, \frac{\pi}{2})$.
\end{lemma}

Therefore, the incident shock interacts with the pseudosonic circle of state $(1)$ to
become a transonic shock.
Moreover, initial-boundary value problem \eqref{equ:1}--\eqref{equ:2}
and \eqref{con:1.1}--\eqref{con:1.2}
in the $(t,\mathbf{x})$-coordinates can be formulated as the following
boundary value problem in the self-similar coordinates $(\xi,\eta)$.

\smallskip
{\bf Problem 2.2} (boundary value problem).
Seek a regular solution $\pSi$ (see Definition \ref{def:regular solution} below) of \eqref{equ:potential flow equation}
in the self-similar
domain $\L$ with the
slip boundary condition on the wedge boundary $\p\L$,
\begin{equation}\label{equ:13}
D\pSi\cdot\nnu|_{\p\L}=0,
\end{equation}
and the asymptotic boundary condition at infinity:
$$
(\r,\pSi)\ra(\bar{\r},\bar{\pSi})
=\bc
(\r_1, \pSi_1)\qquad&\text{in }\{\xi<\xi_1,\
\eta\geq0\},\\[1mm]
(\r_0, \pSi_0)\qquad&\text{in }\{\xi>\xi_1,\
\eta\geq0\}\cup(\{\eta<0\}\cap \Lambda),\ec
$$
when $\xi^2+\eta^2\ra\infty$ in the sense that
$$
\lim_{R\ra\infty}\|\pSi-\bar{\pSi}\|_{C^1(\L\backslash
	B_{R}(0))}=0.
$$

\medskip
The solution of {Problem 2.2} can be shown to be
the solution of {Problem 2.1}.
Therefore, if the nonexistence of the regular solution
of {Problem 2.2} can be shown,
then the nonexistence of
the regular self-similar solution of {Problem 2.1}
follows.

\subsection{Regular solutions}\label{regularsolution}
Since $\overline{\pSi}$ does not satisfy the slip boundary condition
for $\xi\geq0$, the solution must differ from state $(1)$ in
$\{\xi<\xi_1\}\cap\L$ near the wedge corner,
which forces the shock to be
diffracted by the wedge.
Inspired by the ``no diffraction" solution when the wedge is flat $(\sigma=0)$ and based on Lemma \ref{prop2.1 incident shock},
in an open domain $\O$ bounded by shock $\Gamma_{\rm shock}$,
the pseudosonic circle $\Gamma_{\rm sonic}$ of the left state $(1)$
with center $(u_1,0)$ and radius $c_1>0$,
and the cornered wedge $\Gamma_{\rm wedge}$,
the regular solution is expected to be pseudosubsonic and smooth,
to satisfy the slip boundary
condition along the wedge, and to be $C^{1,1}$-continuous across
the pseudosonic circle to become pseudosupersonic.

Let $\mathcal{C}$ be the corner of the wedge. Let $\phi:=\pSi+\frac{1}{2}(\xi^2+\eta^2)$.
Note that $\phi_0\equiv 0$.
To show the nonexistence of regular solutions of {Problem 2.2},
we will show the nonexistence of regular solutions
of the following free boundary problem.

\medskip
{\bf Problem 2.3} (free boundary value problem) (see Fig. \ref{fig:freeboundary}).
Seek a function $\phi\in{\rm Lip}({\overline\Omega})\cap C^1(\overline \Omega\backslash\mathcal{C})
\cap C^3(\overline\Omega\backslash(\overline{\gammashock}\cup
\overline{\Gamma_{\text{\rm sonic}}}\cup\mathcal{C}))$ that satisfies
\begin{equation}
\left(c^2-(\phi_\xi-\xi)^2\right)\phi_{\xi\xi}-2(\phi_\xi-\xi)(\phi_\eta-\eta)\phi_{\xi\eta}
+\left(c^2-(\phi_\eta-\eta)^2\right)\phi_{\eta\eta}=0 \quad  \text{ in $\Omega$},
\label{quasilinearequationofphi}
\end{equation}
with the following boundary conditions.
\begin{enumerate}
\item[\rm (i)] {\it Boundary conditions on the shock}:
\begin{align}\label{bou:RHregular}
&\phi=0\qquad\text{ on  } \Gamma_{\text{\rm shock}};\\
\label{bou1:RHregular}
&\rho(D\varphi, \varphi)D\varphi\cdot\nnu=\rho_0 D\varphi_0\cdot\nnu\qquad\text{ on  } \Gamma_{\text{\rm shock}},
\end{align}
where $\varphi:=\pSi-\frac{1}{2}(\xi^2+\eta^2)$.
\item[\rm (ii)] {\it The sonic circle is a weak discontinuity}:
\[
\phi=\phi_1, \quad D\phi=D\phi_{1}=(u_1,0)\qquad\,\,  \text{on $\overline{\Gamma_{\text{\rm sonic}}}$}.
\]
\item[\rm (iii)] {\it Slip boundary condition along the wedge}:
\begin{equation}\label{equ:40}
	D\phi\cdot \nnu=0   \qquad\,\,    \text{ on $\Gamma_{\text{\rm wedge}}\backslash\MC$}.
	\end{equation}
\end{enumerate}
Moreover, the function $\phi$ is subsonic in $\Omega\cup\Gw^0$, \emph{i.e.},
\begin{equation}
c^2-|D\varphi|^2>0 \qquad \text{ in $\Omega\cup\Gw^0$},
\label{subsonicityinOmega}
\end{equation}
where $\Gw^0$ is the relative interior of $\Gw$.

\begin{figure}
\centering
\includegraphics[width=6cm]{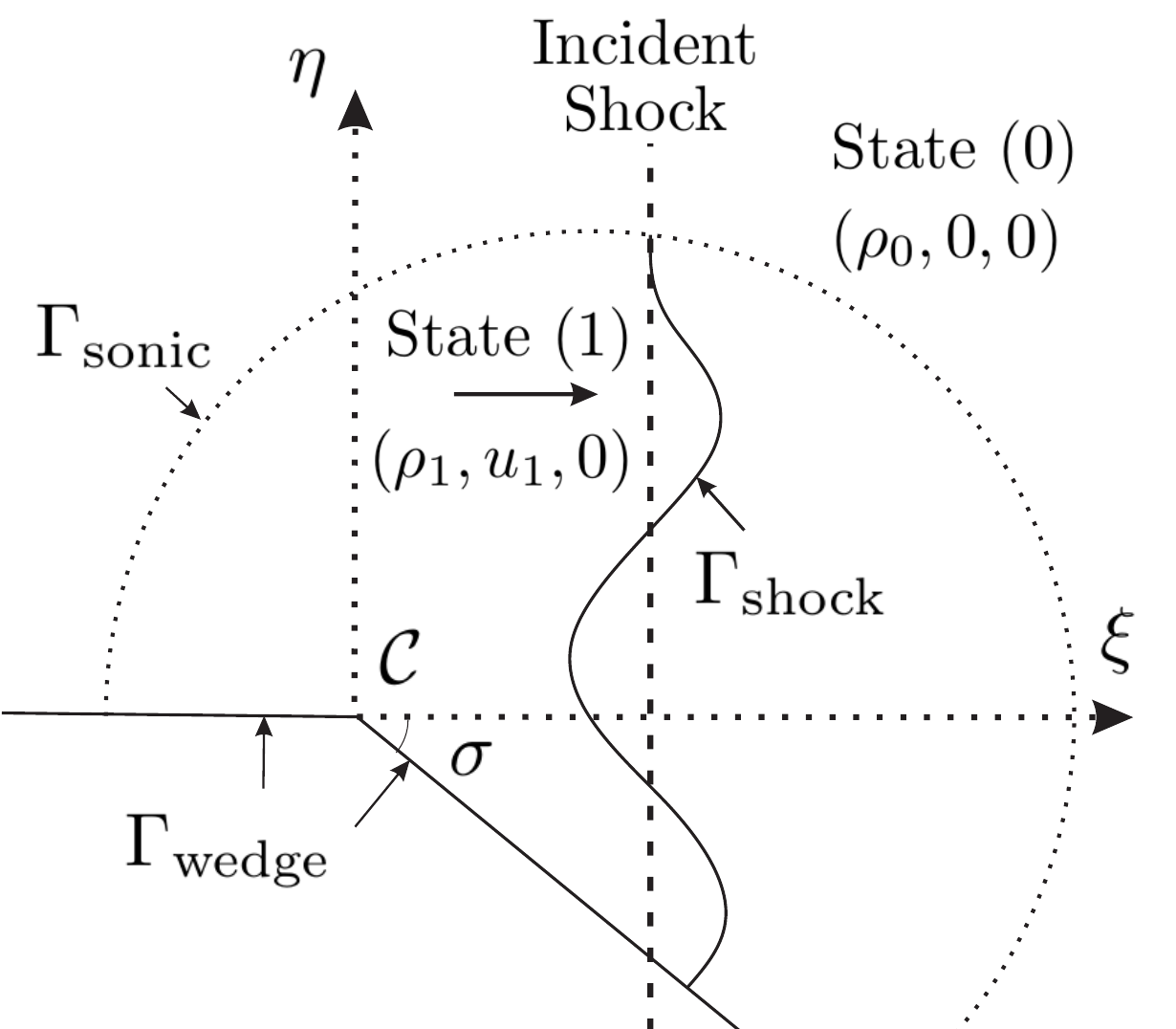}
\caption{Free boundary value problem}\label{fig:freeboundary}
\end{figure}

\medskip
Since the sonic circle is a weak discontinuity, any solution of {Problem 2.3} is a solution of {Problem 2.2},
provided that
$$
\pSi
=\bc
\pSi_1\qquad&\text{in }\{\xi<\xi_1,\
(\xi-u_1)^2+\eta^2\geq c_1^2\},\\[1mm]
\pSi_0\qquad&\text{in }\Lambda\backslash(\{\xi<\xi_1,\
(\xi-u_1)^2+\eta^2\geq c_1^2\}\cup\Omega).\ec
$$

Now we define the notion of regular solutions of the Lighthill problem. We use the notation
$\sigma=\pi-\theta_{\rm w}$.

\begin{definition}\label{def:regular solution}
A function $\phi$ is called a regular solution of the Lighthill problem if $\phi$ is a solution
of {Problem 2.3}
with the following additional condition{\rm :}
There exists $\tsigma\in (\sigma, \pi)$ such that vector $\vg:=(\sin \tsigma, \cos\tsigma)$
is not tangential to $\Gamma_{\text{\rm shock}}$ at any point on $\Gamma_{\text{\rm shock}}$.
\end{definition}

\begin{remark}\label{rem-2.1}
	Note that, by the $C^1$-regularity of $\phi$ in $\overline\Omega\backslash\mathcal{C}$,
	\eqref{bou:RHregular}, and the condition of Definition {\rm \ref{def:regular solution}},
	it follows that $\Gamma_{\text{\rm shock}}$ is a $C^1$-curve up to its ends. Also, in the proof of the main result of this paper, Theorem \ref{thm}, the boundary condition (\ref{bou1:RHregular}) is not used, i.e. we use only (\ref{bou:RHregular}) on $\Gamma_{\text{\rm shock}}$.
\end{remark}

For the condition of Definition \ref{def:regular solution}, we have the following remarks.

\begin{remark}
	When the shock is convex as observed in the experimental results,
	the additional condition in Definition {\rm \ref{def:regular solution}}
	follows directly.
\end{remark}

\begin{remark}\label{rem:2.3}
When the wedge is flat {\rm (}$\sigma=0${\rm )}, there is ``no diffraction" solution,
{\it i.e.}, $\pSi$ consists of only two uniform
states $\pSi_0$ and $\pSi_1$ separated by the incident shock.
Then, when the wedge is almost flat {\rm (}{\it i.e.}, $\sigma$ is very small{\rm )},
one expects that $\pSi$ should be close to $\pSi_1$ in the sense of $C^1(\Omega)$.
In this case, from the Dirichlet condition
\eqref{bou:RHregular},
the curved shock $\Gsh$ is a $C^1$-perturbation of the incident shock.
This implies the condition of Definition {\rm \ref{def:regular solution}}.
\end{remark}

\begin{remark} If \Gammashock is a graph of a $C^1$-function of $\eta$ and the
wedge angle is obtuse {\rm (}{\it i.e.}, $\sigma<\frac{\pi}{2}${\rm )},
then the additional condition in Definition {\rm \ref{def:regular solution}}
holds, since we can choose
$\tsigma=\frac{\pi}{2}$.
\end{remark}

Now we can present our main theorem.

\begin{theorem}[{\bf main theorem}]\label{thm}
The Lighthill problem does not permit a regular solution
in the sense of Definition {\rm \ref{def:regular solution}}.
\end{theorem}

With the estimates and properties developed in \S  \ref{conclusion}
and Appendix A,
the main theorem (Theorem \ref{thm}) will be proved in \S \ref{conclusion}.

Based on Remark \ref{rem:2.3}, one can obtain the following corollary from Theorem \ref{thm}.
\begin{corollary}
The ``no diffraction" solution is not stable in the sense of $C^1(\Omega)$ with respect to the wedge angle as $\sigma\rightarrow0$.
\end{corollary}

\section{Proof of the main theorem --- Theorem \ref{thm}}\label{conclusion}

In order to prove  the main theorem,  we need to derive several estimates on the derivatives of regular
solutions in the sense of Definition \ref{def:regular solution},
based on Lemma \ref{lem:continuitylemma}.
Let $r_0:=\frac{1}{2}\textrm{dist}(\mathcal{C},\Gamma_{\rm shock}\cup\Gamma_{\rm sonic})$.

\begin{lemma}\label{lem:interiorestimate}
	Let $\phi$ be a regular solution in the sense of Definition {\rm \ref{def:regular solution}}.
	Assume that, for any $r\in(0,r_0)$,
	\[
	\|D \phi\|_{C(B_r(\MC)\cap\Omega)}\leq \omega(r),
	\]
	where $\omega(r)$ is a nondecreasing continuous function of $r$ on $[0,r_0]$.
	Then, for any $r\in(0,r_0)$,
	\[
	|D^2\phi|\leq C\frac{\omega(2r)}{r}\qquad\,\, \text{on } \partial B_r\cap\Omega,
	\]
	where $C$ depends only on the elliptic ratio $\lambda$ and bound $\|D\phi\|_{0;B_{r_0}(\MC)\cap\Omega}$.
\end{lemma}

\begin{proof}
	For any fixed $r\in(0,r_0)$,
	we scale $B_r(\MC)$ to $B_2(\MC)$ by the change of coordinates as follows:
	$$
	(x,y):={2 \over r}(\xi,\eta),\qquad \psi(x,y)=\phi(\xi,\eta).
	$$
	
	In the $(x,y)$-coordinates, $\psi$ satisfies
	\begin{equation}\label{linearequationofpsi}
	\sum_{i,j=1}^2 \tilde a_{ij}\psi_{ij}=0\qquad\,\, \mbox{ in } B_2(\MC)\backslash(B_1(\MC)\cup W),
	\end{equation}
	where $\tilde a_{ij}(x,y)=a_{ij}(\frac{r}{2}x, \frac{r}{2}y)$.
	Equation \eqref{linearequationofpsi} is uniformly elliptic in $B_2(\MC)\backslash(B_1(\MC)\cup W)$.
	Then we can apply Nirenberg's estimate ({\rm cf}. Chapter 12 of \cite{GT} or Nirenberg \cite{Ni})
	to obtain
	\begin{equation}
	\|D\psi\|_{C^{\alpha}(B_{7/4}(\MC)\backslash (B_{5/4}(\MC)\cup W))}
	\leq C\|D\psi\|_{L^{\infty}(B_2(\MC))}\leq C{r\over 2}\|D\phi\|_{C(B_r(\MC)\cap\Omega)}
	\leq \frac{C}{2} r\omega(r),
	\label{nirenburgholdergradientestimate}
	\end{equation}
	where $C$ depends only on the elliptic ratio of \eqref{linearequationofpsi}.
	In the $(x,y)$-coordinates, \eqref{linearequationofpsi} can be rewritten as
	\begin{equation}
	\left(c^2-\left(\frac{\psi_x}{r}-rx\right)^2\right)\psi_{xx}-2\left({\psi_x \over r}-rx\right)\left({\psi_y \over r}-ry\right)\psi_{xy}
	+\left(c^2-\left({\psi_y \over r}-ry\right)^2\right)\psi_{yy}=0,
	\label{quasilinearequationofpsi}
	\end{equation}
	where the sonic speed $c$ can be represented as
	\begin{equation}
	c^2=c^2_0-(\gamma-1)\left(\psi-\psi_x x-\psi_y y+{\psi^{2}_{x}+\psi^2_y \over 2 r^2}\right).
	\label{sonicspeedexpressedbypsi}
	\end{equation}
	
	With (\ref{nirenburgholdergradientestimate})--(\ref{sonicspeedexpressedbypsi}),
	we obtain the estimate of the H\"older norm of $\tilde a_{ij}$ as
	\begin{align*}
	\|\tilde a_{ij}\|_{C^{\alpha}(B_{7/4}(\MC)\backslash(B_{5/4}(\MC)\cup W))}
	\leq & \; C(\lambda,\frac{1}{r}\|D\psi\|_{C^{\alpha}(B_{7/4}(\MC)\backslash(B_{5/4}(\MC)\cup W))})\\
	\leq & \;  C(\lambda,\|D \phi\|_{C(B_{r_0}(\MC)\cap\Omega)}),
	\end{align*}
	where $C(\lambda,\|D \phi\|_{C(B_{r_0}(\MC)\cap\Omega)})$ depends only on the elliptic ratio $\lambda$
	and on the bound $\|D \phi\|_{C(B_{r_0}(\MC)\cap\Omega)}$.
	Then we can apply the Schauder estimate to (\ref{linearequationofpsi}) to obtain
	\begin{align*}
	\|D^2\psi\|_{C(B_{13/8}(\MC)\backslash (B_{11/8}(\MC)\cup W))}
	\leq  C\|D \psi\|_{C(B_2(\MC)\backslash (B_1(\MC)\cup W))}
	\leq C\omega(r)r,
	\end{align*}
	where $C$ is a universal constant that depends only on the elliptic ratio $\lambda$ and bound $\|D\phi\|_{0;B_{r_0}(\MC)\cap\Omega}$
	and may be different at each occurrence.
	Scaling back to the $(\xi,\eta)$-coordinates, we have
	\[
	\|D^2\phi\|_{C(\partial B_{3r/4})}\leq C{\omega(r)\over r}.
	\]
	This completes the proof.
\end{proof}

Based on Lemma \ref{lem:interiorestimate} and Lemma \ref{lem:continuitylemma}, we have the following lemma.

\begin{lemma}\label{lem:continuityofvelocity}
If $\phi$ is a regular solution in the sense of Definition {\rm \ref{def:regular solution}},
then
$$
\phi\in C^1(\overline{B_{r_0}(\MC)\cap \Lambda}).
$$
\end{lemma}

\begin{proof} Let $v_1=\phi_{\eta}$ and $v_2=\sin\sigma\phi_{\xi}+\cos{\sigma}\phi_{\eta}$.
Similar to
\eqref{gammaderivativeofequationofphi}--\eqref{equ:43},
we can derive the equation and the boundary conditions of $v_1$ and $v_2$, respectively, near the wedge corner.
The only difference is that $v_1=0$ on $\Gamma_{\rm wedge}\cap\{\eta=0\}$
and $v_2=0$ on $\Gamma_{\rm wedge}\cap\{\eta<0\}$.

Without loss of the generality, let us only consider the case for $w:=v_1=\phi_{\eta}$,
since the argument for $v_2$ is similar.
Then, in order to apply Lemma \ref{lem:continuitylemma}, we must show
\begin{equation}\label{4.1a}
|b_i|\leq\frac{C}{r},
\end{equation}
where $r$ is the distance to the wedge corner.
It suffices to show
$$
|D^2\phi|\leq\frac{C}{r},
$$
which can be achieved by Lemma \ref{lem:interiorestimate} and \eqref{subsonicityinOmega}.
More precisely,
for any $r\in(0,r_0)$,
\[
|D^2\phi|\leq\frac{M(\|D\phi\|_{C^0(\Omega)})}{r}\qquad\,\, \mbox{on }\partial B_r(\MC)\cap\Omega.
\]

By a straightforward calculation, we have
\begin{align*}
0=&\; \left(\phi_{\xi\xi}+{2 a_{12}\over a_{11}}\phi_{\xi\eta}+{a_{22}\over a_{11}}\phi_{\eta\eta}\right)_\eta\\
  =&\; w_{\xi\xi}+{2a_{12}\over a_{11}}w_{\xi\eta}+{a_{22}\over a_{11}}w_{\eta\eta}
  +\left(\frac{2a_{12}}{a_{11}}\right)_{\eta}w_\xi+\left(\frac{a_{22}}{a_{11}}\right)_\eta w_\eta,
\end{align*}
where $a_{11}=c^2-(\phi_{\xi}-\xi)^2$, $a_{12}=-(\phi_{\xi}-\xi)(\phi_{\eta}-\eta)$,
$a_{22}=c^2-(\phi_{\eta}-\eta)^2$, and $|(\frac{2a_{12}}{a_{11}}, \frac{a_{22}}{a_{11}})_\eta|\leq C|D^2\phi|$.
Therefore, $w$ satisfies the equation of form
$$
\sum_{i,j=1}^2a_{ij}w_{ij}+\sum_{i=1}^2b_i w_i=0
$$
with the estimate that $|b_i(\xi,\eta)|\leq {C\over r}$.  This completes the proof.
\end{proof}

\begin{lemma}\label{lem:vanishingofvelocity}
If $\phi$ is a regular solution in the sense of Definition {\rm \ref{def:regular solution}},
then $D\phi=(0,0)$ at the wedge corner.
\end{lemma}

This lemma follows from Lemma \ref{lem:continuityofvelocity}, the boundary condition $\phi_{\nnu}$ on $\Gw\backslash\MC$,
and the fact that the unit normal on $\Gamma_{\rm wedge}\cap\{\eta=0\}$ is different from the one
on $\Gamma_{\rm wedge}\cap\{\eta<0\}$.

\smallskip
By Lemmas \ref{lem:continuityofvelocity}--\ref{lem:vanishingofvelocity},
there exists $f(r)$ that is a nondecreasing continuous function on $[0,1]$
with $f(r)\rightarrow 0$ as $r\rightarrow 0$ such that, for any $r\in(0,r_0)$,
\begin{equation}
\|D \phi\|_{C(B_r(\MC)\cap\Omega)}\leq f(r).
\label{functiondecayestimate}
\end{equation}
Then, by Lemma \ref{lem:interiorestimate}, we have the following lemma.

\begin{lemma}\label{lem:interiorsecondorderderivativeestimateforpotentialfunction}
If $\phi$ is a regular solution in the sense of Definition {\rm \ref{def:regular solution}},
then, for any $r\in(0,r_0)$,
\begin{equation}
 |D^2\phi|\leq M(\|D\phi\|_{C(B_{r_0}(\MC)\cap\Omega)})\frac{f(2r)}{r} \qquad\,\,\text{on } \partial B_r(\MC)\cap\Omega,
\end{equation}
where $f(r)$ is a nondecreasing continuous function on $[0,r_0)$ with $f(r)\rightarrow 0$ as $r\rightarrow 0$.
\end{lemma}

\medskip
We now prove the following key estimate for the sign of a special directional velocity,
which is the weak maximum principle for the directional velocity.

\begin{lemma}\label{lem:maximum}
Let $\phi\in C^3(\overline\Omega\backslash(\overline\gammasonic\cup\overline{\gammashock}\cup\MC))\cap C^1({\overline\Omega})$
be a nonconstant solution of the elliptic equation
\begin{equation}
a_{11}\phi_{\xi\xi}+2a_{12}\phi_{\xi\eta}+a_{22}{\phi_{\eta\eta}}=0  \qquad  \text{ in } \Omega,
\label{linearequationofphi}
\end{equation}
where $a_{ij}\in C^1(\overline\Omega\backslash(\overline\gammasonic\cup\overline{\gammashock}\cup\MC)), i,j=1,2$.
Assume that \eqref{linearequationofphi} is uniformly elliptic near the wedge corner
and strictly elliptic in $\Omega\cup\Gw^0$,
and $\phi_{\nnu}:=D\phi\cdot\nnu=0$ on $\Gamma_{\text{\rm wedge}}$.
Let $\vg:=(\sin \tsigma, \cos\tsigma)$ for $\tsigma\in (\sigma, \pi)$.
If
\[
\phi_{\vg}:=D\phi\cdot\vg \geq0 \qquad  \text{ on }\gammashock\cup\Gamma_{\rm sonic},
\]
then
\[
\phi_{\vg}>0 \qquad \text{ in } (\Omega\cup\gammawedge)\backslash\MC.
\]
\end{lemma}

\begin{proof} We divide the proof into three steps.

\smallskip
{1}. Let $w:=\phi_{\vg}$.
For any $\varepsilon>0$, by assumption,
we can see that $w\geq-\varepsilon$ in a neighborhood of
$\Gamma_{\text{\rm shock}}\cup \Gamma_{\text{\rm sonic}}$,
since $\phi\in C^1({\overline\Omega\backslash\mathcal{C}})$.
This $\varepsilon$ can be arbitrarily small
if the neighborhood is chosen to be very small.

Note that $a_{ij}\in C^1(\overline\Omega\backslash(\overline\gammasonic\cup\overline{\gammashock}\cup\MC))$.
We can differentiate (\ref{linearequationofphi}) to obtain
\begin{align}
0 &=\big({a_{11}\over a_{22}}\phi_{\xi\xi}+{2a_{12}\over a_{22}}\phi_{\xi\eta}+\phi_{\eta\eta}\big)_\vg \notag\\
  &={a_{11}\over a_{22}} w_{\xi\xi}  +{2a_{12}\over a_{22}}w_{\xi\eta} +w_{\eta\eta}
  +\big({a_{11}\over a_{22}}\big)_\vg \phi_{\xi\xi}+\big({2a_{12}\over a_{22}}\big)_\vg \phi_{\xi\eta}.
\label{gammaderivativeofequationofphi}
\end{align}

Then, using again (\ref{linearequationofphi}),
we can write $\phi_{\xi\eta}$ and $\phi_{\xi\xi}$ as a linear combination of $w_\xi$ and $w_\eta$.
More precisely, we have
\begin{align}
\begin{pmatrix}
\phi_{\xi\xi}\\
\phi_{\xi\eta}
\end{pmatrix}
=J^{-1}
\begin{pmatrix}
\sin\tsigma-\frac{2a_{12}}{a_{22}}\cos\tsigma&-\cos\tsigma\\ \frac{a_{11}}{a_{22}}\cos\tsigma&\sin\tsigma
\end{pmatrix}
\begin{pmatrix}
w_\xi\\w_\eta
\end{pmatrix},
\label{secondorderderivativebyderivativeofz}
\end{align}
where $J=\sin^2\tsigma-\frac{a_{12}}{a_{22}}\sin(2\tsigma)+\frac{a_{11}}{a_{22}}\cos^2\tsigma>0$
in $\overline\Omega\backslash(\overline\gammasonic\cup\overline{\gammashock}\cup\MC)$,
thanks to the ellipticity of \eqref{linearequationofphi}.

Plugging (\ref{secondorderderivativebyderivativeofz}) into (\ref{gammaderivativeofequationofphi}),
by a straightforward calculation, we find that
$w$ satisfies a linear strictly elliptic equation with  coefficients  locally bounded in $\Omega$ and without zeroth order term.
Thus, $w$ cannot achieve the local minimum anywhere in $\Omega$ by the maximum principle.

\smallskip
{2}. Now we show that $w$ cannot achieve the local minimum on  $\Gamma_{\rm wedge}$.

First,  $\phi_\nnu=\phi_\eta=0$ on $\Gamma_{\text{\rm wedge}}^+:=\Gamma_{\text{\rm wedge}}\cap\{\eta=0\}$.
Taking the tangential derivative along the wedge boundary, we have
$$
\phi_{\xi\eta}=0\qquad\, \mbox{on }\Gamma_{\rm wedge}^+.
$$
Plugging it into (\ref{secondorderderivativebyderivativeofz}) yields
\begin{equation}\label{equ:42}
w_{\boldsymbol{\beta}_+}=0\qquad\, \mbox{on }\Gamma_{\rm wedge}^+,
\end{equation}
where $\boldsymbol{\beta}_+:= (a_{11}\cos\tsigma, a_{22}\sin\tsigma)$.
Notice that $\tsigma\in (\sigma, \pi)$ so that
\[
\boldsymbol{\beta}_+\cdot \nnu=a_{22}\sin\tsigma\neq 0.
\]

Next, on  $\Gamma_{\text{\rm wedge}}^-:=\Gamma_{\text{\rm wedge}}\cap\{\eta<0\}$, the argument is similar. We rotate coordinates by angle $\sigma$ clockwise so that, in the new coordinates, the boundary part
$\Gamma_{\text{\rm wedge}}^-$ lies on the positive $\xi$-axis, i.e. $\nu=(0,1)$ on $\Gamma_{\text{\rm wedge}}^-$. Equation
(\ref{linearequationofphi}) in the rotated coordinates has the similar form with the new coefficients $\hat a_{ij}$, satisfying the ellipticity
condition with the same constants. Vector $\vg$ in the rotated coordinates is
$\vg:=(\sin (\tsigma-\sigma), \cos(\tsigma-\sigma))$. Now
we can repeat the previous calculation working the rotated coordinates and obtain
\begin{equation}\label{equ:43}
w_{\boldsymbol{\beta}_-}=0\qquad\, \mbox{on }\Gamma_{\rm wedge}^-,
\end{equation}
where $\boldsymbol{\beta}_-:= (\hat a_{11}\cos(\tsigma-\sigma), \hat a_{22}\sin(\tsigma-\sigma))$ in these coordinates. Noting that  $\tsigma-\sigma\in (0, \pi)$, we have
\[
\boldsymbol{\beta}_-\cdot\boldsymbol{\nu}=\hat a_{22}\sin(\tsigma-\sigma)\neq0   \qquad\, \mbox{on }\Gamma_{\rm wedge}^-.
\]
Therefore, the boundary conditions \eqref{equ:42}--\eqref{equ:43}
are oblique. Then, by Hopf's lemma,
we conclude that $w$ cannot achieve its minimum on $\Gamma_{\text{wedge}}$, except $\mathcal{C}$.

\smallskip
{3}. By the $C^1(\overline{\Omega})$-regularity of $\phi$,
we see that $D\phi=0$ at $\mathcal{C}$.
This means that $w=0$ at $\mathcal{C}$.
Thus, we finally obtain
\begin{equation}\label{equ:44}
w\geq-\epsilon \qquad \text{in $\overline\Omega$}.
\end{equation}

\medskip
Let $\varepsilon\rightarrow0$. Then we have
$$
w\geq0\qquad \mbox{in } \overline\Omega.
$$
Therefore, by the strong maximum principle,
if $\phi$ is not a constant, then $w>0$ in $(\Omega\cup\gammawedge)\backslash\MC$.
This completes the proof.
\end{proof}

Based on Lemma \ref{lem:maximum}, we have the following corollary for the regular solutions
defined in Definition \ref{def:regular solution}.

\begin{corollary} Let $\vg$ be as in Definition {\rm \ref{def:regular solution}} for regular solutions.
Then
\begin{equation}
\phi_{\vg}>0 \qquad \text{ in } (\Omega\cup\gammawedge)\backslash\MC.
\label{positivityofz}
\end{equation}
\end{corollary}

\begin{proof}
Equation \eqref{linearequationofphi} is strictly elliptic in $\Omega\cup\Gw^0$ due to (\ref{subsonicityinOmega}). We also note that Hopf's lemma can be applied at every point of
$\Gamma_{\text{\rm wedge}}$ including the corner point $\MC$ because the interior angle
at $\MC$ is $\pi+\sigma >\pi$, i.e., the interior ball condition holds at $\MC$, and $\phi$
is $C^1$ up to $\MC$ by Lemma \ref{lem:vanishingofvelocity}.
Thus, by the strong maximum principle and Hopf's lemma,
$$
\sup_\Omega \phi=\sup_{\partial \Omega\backslash \Gamma_{\text{\rm wedge}}} \phi=
\sup_{\overline\gammasonic\cup\overline{\gammashock}} \phi.
$$

On $\overline\gammasonic$, we have $\phi=\phi_1$  by condition (ii) of Problem 2.3. Note that
$\Omega\subset \{\xi < \xi_1\}$ as discussed at the beginning of \S \ref{regularsolution}, and thus
$\gammasonic\subset \{\xi < \xi_1\}$. Also, $\phi_1<\phi_0=0$ in $\{\xi < \xi_1\}$. Thus
$\phi<0$ on $\gammasonic$.

By the boundary condition \eqref{bou:RHregular}, $\phi=0$ on $\Gsh$.

This shows that $\max_{\overline\Omega} \phi=0$, and the maximum value $\phi=0$ is attained at every point of $\gammashock$. From this,
\begin{equation}\label{phiOnShock}
D\phi=-\phi_\nnu \nnu  \qquad\,\, \text{ and } \;\;\phi_\nnu \le 0\qquad\,\, \text{ on }\gammashock.
\end{equation}

Also, recall that $\gammashock$ is $C^1$ up to its endpoints by Remark \ref{rem-2.1}, and $\phi\in C^1(\overline\Omega)$
by conditions of Problem 2.3 and Lemma \ref{lem:continuityofvelocity}. Then, at the point $A$ of intersection of $\overline\gammashock$ with $\overline\gammasonic$, we have $D\phi=D\phi_1$
 by condition (ii) of Problem 2.3, which shows that the interior (for $\Omega$) unit normal to $\gammashock$ at $A$
 is $\nnu(A)=\frac{D(\phi_0-\phi_1)}{|D(\phi_0-\phi_1)|}=(-1, 0)$, where we used (\ref{2.11a}) in the last equality.  Thus $\nnu(A)\cdot \vg = -\sin \tsigma <0$, where we have used that $\tsigma\in (0, \pi)$.
Since $\Gamma_{\text{\rm shock}}$ is not tangential to $\vg$ at any point on $\Gsh$ and $\phi\in C^1{(\overline\Omega\backslash\MC)}$,
we have
\[\nnu\cdot \vg < 0  \qquad\,\, \text{ on }\gammashock. \]
Combining this with (\ref{phiOnShock}), we obtain
\[
w:= \phi_{\vg}\geq0 \qquad\,\, \text{ on }\gammashock.
\]

On $\Gamma_{\text{\rm sonic}}$, $w=(u_1,0)\cdot \vg=u_1\sin\tsigma>0$.

Finally, since $\phi\in C^3(\overline\Omega\backslash(\overline{\Gamma_{\text{\rm sonic}}}\cup\overline{\gammashock}\cup\MC))$,
$a_{ij}\in C^1(\overline\Omega\backslash(\overline\gammasonic\cup\overline{\gammashock}\cup\MC))$.
Then Lemma \ref{lem:maximum} applies.
\end{proof}

Based on these lemmas, we are now going to establish the main theorem, Theorem \ref{thm}.

\medskip
\begin{proof}[\bf Proof of the main theorem --- Theorem {\rm \ref{thm}}.]
We prove the main theorem by deriving a direct contradiction to (\ref{positivityofz}).

Since the Euler equations are invariant with respect to the Galilean transformation,
we rotate and reflect the coordinates such that $\vg$ in the new coordinates
is the horizontal direction pointing to the left,
\emph{i.e.},  $\mathcal{R}_{\ast}\vg=-\partial_\xi$,
where $\mathcal{R}$ is the corresponding rotation and reflection operator.
After this transformation, $\phi$ still satisfies a uniform elliptic equation
of second order in $(B_r(\MC)\cap\Omega)\backslash\MC$ for any $r\in(0,r_0)$.
For notational simplicity, we still use $(\xi,\eta)$ as the coordinates after
this transformation. Furthermore, we use notation $(u,v)=(\phi_\xi, \phi_\eta)$.

As before, let $w:= \phi_{\vg}$.
Then, as shown in Fig. \ref{fig:PotentialFunctionphcannotbeMonotonicalonggammaDirection},
$w$ defined on the left-hand side of
Fig. \ref{fig:PotentialFunctionphcannotbeMonotonicalonggammaDirection}
is transformed into $-u=-\phi_\xi$
corresponding to  that on the right-hand side of
Fig. \ref{fig:PotentialFunctionphcannotbeMonotonicalonggammaDirection}.
Moreover, the result from \eqref{positivityofz} in the old coordinates that
$w\geq\delta$ on $\partial B_r(\MC)\cap\Omega$
is transformed into
\[
u\leq-\delta \qquad \text{ on } \partial B_r(\MC)\cap\Omega
\]
in the new coordinates.

\begin{figure}
\includegraphics[height=4.1cm]{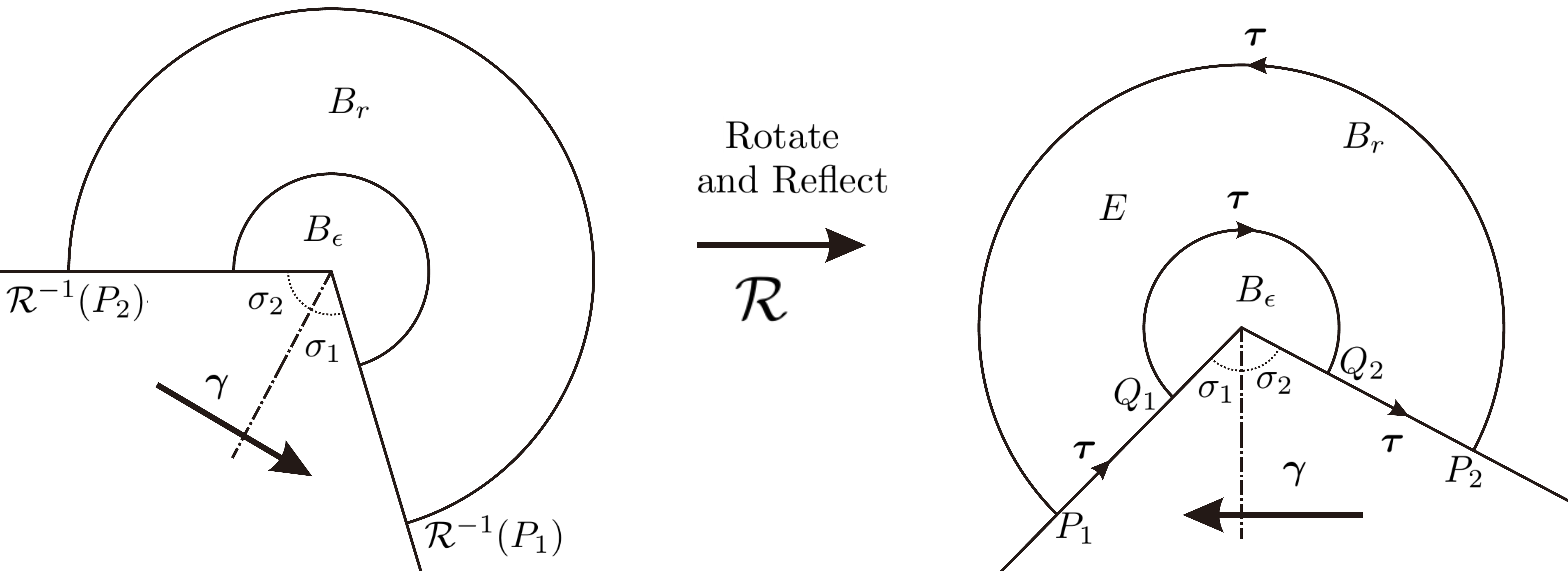}
\caption{The potential function $\phi$ cannot be monotonic in the $\boldsymbol{\gamma}$ direction}
\label{fig:PotentialFunctionphcannotbeMonotonicalonggammaDirection}
\end{figure}

As shown in Fig. \ref{fig:PotentialFunctionphcannotbeMonotonicalonggammaDirection},
the dashed line is defined to pass corner $\MC$ and to be perpendicular to $\vg$.
The angles between the dashed line and the two sides of \Gammawedge are $\sigma_1$ and $\sigma_2$,
where both $\sigma_1$ and $\sigma_2$ belong to $(0, \pi-\sigma)$.
In fact, $\tilde\sigma$
in Definition \ref{def:regular solution} equals to $\sigma+\sigma_1=\pi-\sigma_2$.

Now we start to work in the new coordinates and derive a contradiction
from an integral estimate of $u$.

Define
\[
h=(u+\delta)^+.
\]
Since $u<-\delta$ on $\partial B_r(\MC)\cap\Omega$, we have
$$
h=0\qquad\text{ on } \partial B_r(\MC)\cap\Omega.
$$
Moreover, it is direct to see that
$$
h\leq\delta+\|D\phi\|_{C(\Omega)}\qquad \text{ in  }B_r(\MC)\cap \Omega.
$$
For any $\epsilon\in(0,{r \over 2})$, define domain $E$ as
$$
E:=(B_r(\MC)\backslash B_\epsilon(\MC))\cap\Omega
$$
with vertices
\begin{align*}
Q_1=\partial B_\epsilon(\MC)\cap \Gamma_{\text{\rm wedge}}\cap\{\xi<0\},
  \qquad &Q_2=\partial B_\epsilon(\MC)\cap \Gamma_{\text{\rm wedge}} \cap\{\xi>0\},\\
P_1=\partial B_r(\MC)\cap \Gamma_{\text{\rm wedge}}\cap\{\xi<0\},\qquad
   &P_2=\partial B_r(\MC)\cap \Gamma_{\text{\rm wedge}} \cap\{\xi>0\}.
\end{align*}
Then define $I$ as follows:
\begin{equation}\label{equ:50}
I=\int_{\{u>-\delta\}\cap E}\Big({a_{11}\over a_{22}}u_\xi^2+{2 a_{12} \over a_{22}}u_\xi u_\eta+u_{\eta}^2\Big)d\xi d\eta.
\end{equation}

By a straightforward calculation, we have
\begin{align*}
I&=\int_{E}\Big({a_{11} \over a_{22}}u_\xi h_\xi+{2a_{12} \over a_{22}}u_\eta h_\xi+u_\eta h_\eta\Big) d\xi d\eta\\
 &=\int_E  \Big(\nabla\cdot\Big({a_{11} \over a_{22}}u_\xi h+{2a_{12} \over a_{22}}u_\eta h,u_\eta h\Big)
                             -\nabla\cdot\Big({a_{11}\over a_{22}}u_\xi+{2 a_{12} \over a_{22}} u_\eta,u_{\eta}\Big)h\Big)d\xi d\eta\\
 &=- \int_E   \nabla\cdot\big(v_\eta h, -v_\xi h\big)d\xi d\eta \\
 &=-\int_{\partial E} v_\T h\,ds\\
 &=-\int_{\partial B_\epsilon(\MC)\cap\Omega}  v_\T h\, ds -\int_{\overline{P_1Q_1}} v_\T h\, ds
   -\int_{\overline{Q_2P_2}} v_\T h\,ds,
\end{align*}
where $\T$ is the unit tangential vector of $\partial E$,
with the direction illustrated in Fig. \ref{fig:PotentialFunctionphcannotbeMonotonicalonggammaDirection},
and the line integral is respect to the arc length $s$.
Here, for the third identity, we have used the equation
that $\frac{a_{11}}{a_{22}}u_\xi+\frac{2a_{12}}{a_{22}}u_\eta=-v_\eta$,
while, for the fifth identity, we have used the fact that $h=0$ on $\partial B_r(\MC)\cap\Omega$.

Using the slip boundary condition \eqref{equ:40} on $\Gamma_{\text{\rm wedge}}$, we have
$$
v=\frac{u}{\tan\sigma_1} \qquad \ \ \text{on \ $\overline{P_1Q_1}$}
$$
and
$$
v=-\frac{u}{\tan\sigma_2} \qquad \ \ \text{on \ $\overline{P_2Q_2}$}.
$$
Applying the relations above, we have
\begin{align*}
&-\int_{\overline{P_1Q_1}} v_\T h\,ds -\int_{\overline{Q_2P_2}}  v_\T h\,ds \\
&=-\int_{\overline{P_1Q_1}}  \frac{h_\T h}{\tan\sigma_1}\,ds  +\int_{\overline{Q_2P_2}} \frac{h_\T h}{\tan\sigma_2}\,ds\\
&=-\int_{\overline{P_1Q_1}} \frac{(h^2)_\T}{2\tan\sigma_1}\,ds   +\int_{\overline{Q_2P_2}} \frac{(h^2)_\T}{2\tan\sigma_2}\,ds \\
&=-\frac{h^2(Q_1)}{2\tan\sigma_1}+\frac{h^2(P_1)}{2\tan\sigma_1}-\frac{h^2(Q_2)}{2\tan\sigma_2}+\frac{h^2(P_2)}{2\tan\sigma_2}.
\end{align*}
By Lemma \ref{lem:continuityofvelocity},
$\phi$ is $C^1$ at the wedge corner and then,  using Lemma \ref{lem:vanishingofvelocity},
$D\phi=0$ at the wedge corner.
Thus, $h(Q_1)$ and $h(Q_2)$ converge to $\delta$ as $\epsilon$ converges to zero.
Therefore, as $\epsilon\rightarrow0$,
\begin{equation} \label{gammawedge}
- \int_{\overline{P_1Q_1}}  v_\T h\, ds  -\int_{\overline{Q_2P_2}} v_\T h\, ds
\rightarrow
 -\frac{\delta^2}{2}\left(\frac{1}{\tan\sigma_1}+\frac{1}{\tan\sigma_2}\right)
 =-\frac{\delta^2\sin(\pi-\sigma)}{2\sin\sigma_1\sin\sigma_2}.
\end{equation}

By Lemma \ref{lem:interiorsecondorderderivativeestimateforpotentialfunction},
we conclude that there exists a function $f(\epsilon)$ with the properties that
$f(\epsilon)$ is a nondecreasing continuous function of $\epsilon\in[0,r_0)$,
and $f(\epsilon)\rightarrow 0$ as $\epsilon\rightarrow 0$, such that
$$
|D^2\phi|\leq C{f(2\epsilon) \over \epsilon}\qquad\mbox{on }\partial B_\epsilon(\MC)\cap\Omega.
$$
Then, as $\epsilon\rightarrow 0$,
\begin{equation}
\int_{\partial B_\epsilon(\MC)\cap\Omega}v_\T h\,ds
\leq \big(\delta+\|D\phi\|_{C(\Omega)}\big)\int_{\partial B_\epsilon(\MC)\cap\Omega}|D^2\phi|\,ds
\leq 2\pi\epsilon\; C {f(2\epsilon)\over \epsilon}\rightarrow 0.
\label{innerboundary}
\end{equation}

Estimates (\ref{gammawedge})--(\ref{innerboundary}) together imply that, as $\epsilon\rightarrow 0$,
\[
I\rightarrow -\frac{\delta^2\sin(\pi-\sigma)}{2\sin\sigma_1\sin\sigma_2}<0.
\]

On the other hand, $I$ is nonnegative from \eqref{equ:50} and the ellipticity.
This is a contradiction, so that the regular solution defined in Definition \ref{def:regular solution}
does not exist indeed.
\end{proof}


\appendix
\section{Continuity estimate at the corner}\label{continuity}

In this appendix, we show that a solution of some type of linear elliptic
equations, which is only assumed to be $L^{\infty}$ at the wedge corner $\mathcal{C}$,
is actually continuous.

\begin{lemma}\label{lem:continuitylemma}
	Let $v\in L^\infty(\overline{B_1\cap \Lambda})\cap C^2(\overline{B_1\cap \Lambda}\setminus\mathcal{C})$
	for the wedge corner $\mathcal{C}$
	satisfy
	\begin{align*}
	\sum_{i,j=1}^2a_{ij}v_{ij}+\sum_{i=1}^2 b_iv_i=0 \qquad  &\text{in } B_1\cap \Lambda,\\
	v=0 \qquad & \text{on $\partial\Lambda\cap B_1\cap\{\xi>0\}$}, \\
	v_{\boldsymbol{\beta}}=0\qquad & \text{on    $\partial\Lambda\cap B_1\cap\{\xi<0\}$,}
	\end{align*}
	where $\sum_{i,j=1}^2a_{ij}y_iy_j\ge \lambda|\mathbf{y}|^2$ for each $\mathbf{y}=(y_1,y_2)\in \mr^2$,
	$|a_{ij}|\le \lambda^{-1}$, $|b_i|\leq \frac{C_E}{r}$,  $|\boldsymbol{\beta}\cdot\nnu|>{1\over C_E}$,
	and $|\boldsymbol{\beta}|=1$ for some positive constants $\lambda$  and $C_E$,
	and $r=\sqrt{\xi^2+\eta^2}$ denotes the distance to the wedge corner.
	Then $v$ is continuous in $\overline\MB:=\overline{B_1\cap\Lambda}$.
	Moreover,
	$$
	|v(\xi, \eta)|\le Cr^\alpha \|v\|_{L^\infty(\MB)} \qquad \mbox{in $\MB$},
	$$
	where $C>0$ and $\alpha\in (0,1)$ depend only on $\lambda$ and $C_E$.
\end{lemma}

\begin{remark}
	The proof of this lemma follows the methods in Lieberman \cite{Li}
	and the references therein, more specifically, their versions in \cite{CFbook,EllingLiu2}.
\end{remark}

\begin{proof}
	We use the comparison function considered in  step 3 of the proof of
	Theorem 4.3.16 in \cite{CFbook}.
	
	\newcommand{\bW}{w}
	\newcommand \htL{ L}
	
	We use the polar coordinates $(r, \theta)$ centered at the origin and choose
	the orientation of $\theta$ and the direction of ray $\theta=0$ so that
	\begin{align*}
	&\MB=\{|\theta|\le \hat\theta \}\cap B_1,\\
	&\partial\Lambda\cap B_1\cap\{\xi>0\}=
	\partial\MB\cap\{\theta= -\hat\theta \},\\
	&\partial\Lambda\cap B_1\cap\{\xi<0\}=
	\partial\MB\cap\{\theta= \hat\theta \},
	\end{align*}
	where $\hat\theta :=\frac 12(2\pi-\theta_{\rm w}) \in (0, \pi)$.
	We show the existence of $\bW\in C^\infty(\overline\Omega\setminus\{{\bf 0}\})$
	such that
	\begin{align}\label{PD-Ineq-BarW-corner}
	& \sum_{i,j=1}^2a_{ij}D_{ij}\bW +\sum_{i=1}^2b_iD_i\bW   <0
	\qquad\; \mbox{ in } \MB, \\
	\label{BarW-ONgam2ineq-barrierEL-bc}
	& \bW_{\boldsymbol{\beta}}<0 \qquad\qquad\qquad\qquad\qquad\quad\quad \mbox{ on $\partial\MB\cap\{\theta= \hat\theta \}$}.
	\end{align}
	Then we consider the function
	\begin{equation}\label{def-BarW-corner}
	\bW(r, \theta)=r^{\alpha} \bar h(\theta) \qquad\text{ in } \MB \,\,\,\mbox{with $\bar h(\theta)=1-\mu e^{-\htL\theta}$},
	\end{equation}
	where the constants $\alpha\in (-1,1)$, $\mu\ge 0$, and $\htL\ge 1$
	will be fixed below.
	
	For each $\htL\ge 1$, we choose $\mu=\frac 12 e^{-\pi  \htL}$
	so that $\bar h(\theta)\ge \frac 12$ for all $\theta\in[-\pi, \pi]$.
	We use such a constant $\mu=\mu(L)$ from now on.
	
	Now we show that $\bW$ is a supersolution of the equation after a certain choice of $\htL$, $\alpha$.
	Fix a point $(\xi, \eta)\in\MB$,
	and rotate the Cartesian coordinates in $\mr^2$ to
	become
	the radial and tangential coordinates at $(\xi, \eta)$.
	We denote by $a_{rr}$, $a_{r\theta}$, {\rm etc.}, the coefficients
	of the equation in these rotated coordinates.
	Below, $C$ is a universal constant that may be different at each occurrence,
	depending only on $\lambda$ and $C_E$.
	We use that $\alpha\in(-1,1)$ and $\frac{1}{2}\le h(\theta)\le 1$
	to compute at $(\xi, \eta)$
	\begin{equation*}
	\begin{split}
	\sum_{i,j=1}^2a_{ij}D_{ij}\bW &=a_{rr}\bW_{rr}+2a_{r\theta}(r^{-1}\bW_{r\theta}-r^{-2}\bW_\theta)
	+a_{\theta\theta}(r^{-2}\bW_{\theta\theta}+r^{-1}\bW_r)\\
	&=\Big(\mu \htL e^{-\htL\theta}\big(-a_{\theta\theta}\htL+2a_{r\theta}(\alpha-1)\big)
	+\alpha\big( (\alpha-1)a_{rr}+a_{\theta\theta}\big) \bar h(\theta)\Big) r^{\alpha-2} \\
	&\le\Big(\mu \htL e^{-\htL\theta}\big(-\frac{\htL}C+C\big)
	+C|\alpha|\Big) r^{\alpha-2}.
	\end{split}
	\end{equation*}
	Then, for any $\htL\ge 10 C^2$, using that $\mu=\frac 12 e^{-\pi \htL}>0$,
	we can choose $\hat\alpha>0$ small so that, for any $|\alpha|\le \hat\alpha$,
	\begin{equation*}
	\begin{split}
	&\sum_{i,j=1}^2a_{ij}D_{ij}\bW\le-\mu \frac{\htL^2}{2C}
	e^{-\htL\theta} r^{\alpha-2}.
	\end{split}
	\end{equation*}

	Then we have
	\begin{equation*}
	\begin{split}
	\sum_{i,j=1}^2a_{ij}D_{ij}\bW+\sum_{i=1}^2b_iD_i\bW
	&\le-\mu \frac{\htL^2}{2C} e^{-\htL\theta} r^{\alpha-2}+\frac Cr |D\bW|\\
	&\;\;\le \big(-\mu \frac{\htL^2}{2C} e^{-\htL\theta}+C(\alpha+
	\mu \htL e^{-\htL\theta})\big)r^{\alpha-2}\\
	&\;\;=\big(\mu\htL e^{-\htL\theta}(-\frac{\htL}{2C} +C) +C\alpha\big)r^{\alpha-2}.
	\end{split}
	\end{equation*}
	We first choose $\htL$ large so that $-\frac{\htL}{2C} +C\le -\frac{\htL}{4C}$,
	{\rm i.e.}, $\htL\ge 4C^2$.
	Then, for such $\htL$,
	using that $\mu=\frac 12 e^{-\pi\htL}>0$ and
	$\theta\in[-\pi, \pi]$,  we can choose $\alpha(\htL)$ small
	so that the last expression is negative, {\rm i.e.}, that
	(\ref{PD-Ineq-BarW-corner}) holds.
	
	To show (\ref{BarW-ONgam2ineq-barrierEL-bc}), let
	$\nnu$ be the interior unit normal on $\partial\MB\cap\{\theta= \hat\theta\}$
	with respect to $\MB$, and let $\ttau$
	be the unit tangent vector
	pointing away from the corner.
	Then
	\begin{align*}
	\bW_{\bbeta}&=-\frac {\bbeta\cdot\nnu}{r} \bW_\theta+
	\bbeta\cdot\ttau\,\bW_r\\
	&\le \big(-(\bbeta\cdot\nnu)\mu\htL e^{-\htL\theta}+C|\alpha|\big)r^{\alpha-1}\\
	&\le
	\Big(-\frac{\mu\htL}{C_E} e^{-\htL\theta}+C|\alpha|\Big)r^{\alpha-1}.
	\end{align*}
	
	Therefore, for any $\htL\ge 1$, using that $\mu=\frac 12 e^{- \pi \htL}>0$,
	we can choose $\hat\alpha>0$ small such that, for any $|\alpha|\le \hat\alpha$,
	the last
	expression is negative for any $\theta=\hat\theta\in[0, \pi]$.
	Then (\ref{BarW-ONgam2ineq-barrierEL-bc}) holds.
	
	Now we fix $\htL$ sufficiently large to satisfy all the conditions
	stated above. This also fixes $(\mu, \hat\alpha)$.
	
	Let $\bW^\pm=r^{\pm\hat\alpha} \bar h(\theta)$,
	where $\bar h(\theta)$ is from
	(\ref{def-BarW-corner}).
	Then, using that $\bar h(\theta)\ge \frac 12$,
	\begin{equation}\label{comparFunctLarge}
	\begin{split}
	\bW^+&\ge \frac 12  \qquad\quad\mbox{ on } \partial\MB\cap\partial B_1,\\
	\bW^-&\ge \frac 12 r^{-\hat\alpha} \quad\,\,\mbox{ on } \MB\cap\partial B_r.
	\end{split}
	\end{equation}
	
	For $\varepsilon\in(0,1)$, let
	$$
	V_\varepsilon=2\|v\|_{L^\infty(\MB)}\bW^+ +\varepsilon \bW^-.
	$$
	Then $V_\varepsilon$ satisfies
	(\ref{PD-Ineq-BarW-corner})--(\ref{BarW-ONgam2ineq-barrierEL-bc}).
	Also, by (\ref{comparFunctLarge}) and since $\bW^\pm> 0$ in $\MB$,
	for each $\varepsilon\in (0,1)$, there exists $R_\varepsilon\in (0,1)$ so that, for each
	$R\in (0, R_\varepsilon]$,
	\begin{equation*}
	\begin{split}
	V_\varepsilon&\ge \|v\|_{L^\infty(\MB)}  \qquad\mbox{ on } (\partial\MB\cap\partial B_1)\cup
	(\MB\cap\partial B_R).
	\end{split}
	\end{equation*}
	Then, by the comparison principle, for each
	$R\in (0, R_\varepsilon]$,
	$$
	v\le V_\varepsilon\qquad\mbox{ in } \MB\setminus \overline B_R.
	$$
	By a similar argument,
	$$
	v\ge -V_\varepsilon\qquad\mbox{ in } \MB\setminus \overline B_R.
	$$
	Combining these two estimates together, sending $R\to 0+$ for each $\varepsilon$, and then
	sending $\varepsilon\to 0+$, we obtain
	$$
	|v|\le \bW^+\qquad\mbox{ in } \MB.
	$$
	Then, using that $\frac 12\le \bar h(\theta)\le 1$, we have
	$$
	|v(r, \theta)|\le 2\|v\|_{L^\infty(\MB)} r^{\hat\alpha}  \qquad\mbox{in $\MB$}.
	$$
	This completes the proof.
\end{proof}

~\\ \textbf{Acknowledgements}.
The research of Gui-Qiang G. Chen was supported in part by
the UK
Engineering and Physical Sciences Research Council Award
EP/L015811/1 and the Royal Society--Wolfson Research Merit Award (UK).
The research of Mikhail Feldman was
supported in part by the National Science Foundation under Grant DMS-1764278 and DMS-1401490,
and the Van Vleck
Professorship Research Award  by the University of Wisconsin-Madison,
The research of Jingchen Hu was supported
by China's Scholarship Council (201306340088).
The research of Wei Xiang was supported in part
by the Research Grants
Council of the HKSAR, China (Project CityU 21305215,
Project CityU 11332916, Project CityU 11304817, and Project CityU11303518).

\end{document}